\documentclass[12pt, reqno]{amsart}
\usepackage{amsmath, amsthm, amscd, amsfonts, mathrsfs, amssymb, graphicx, color}
\usepackage[bookmarksnumbered, colorlinks, plainpages]{hyperref}
\input{mathrsfs.sty}

\newtheorem{theorem}{Theorem}[section]
\newtheorem{lemma}[theorem]{Lemma}

\newtheorem{corollary}[theorem]{Corollary}
\theoremstyle{definition}

\newtheorem{example}[theorem]{Example}

\theoremstyle{remark}
\newtheorem{remark}[theorem]{Remark}
\numberwithin{equation}{section}

\def\R{\mathbb{R}}

\begin{document}

\title{Operator inequalities of Jensen type}

\author[M.S. Moslehian, J. Mi\'{c}i\'{c}, M. Kian]{M. S. Moslehian$^1$, J. Mi\'{c}i\'{c}$^2$ and M. Kian$^3$}

\address{$^1$ Department of Pure Mathematics, Center of Excellence in Analysis on Algebraic Structures (CEAAS), Ferdowsi University of Mashhad, P.O. Box 1159, Mashhad 91775, Iran}
\email{moslehian@um.ac.ir and moslehian@member.ams.org}

\address{$^2$ Faculty of Mechanical Engineering and Naval Architecture, University of Zagreb,
Ivana Lu\v ci\' ca 5, 10000 Zagreb, Croatia} \email{jmicic@fsb.hr}

\address{$^3$ Department of Mathematics, Faculty of Basic Sciences, University of Bojnord, Bojnord,
Iran} \email{kian@member.ams.org}



\begin{abstract}
We present some generalized Jensen type operator inequalities
involving sequences of self-adjoint operators. Among other things,
we prove that if $f:[0,\infty) \to \mathbb{R}$ is a continuous
convex function with $f(0)\leq 0$, then
   \begin{equation*}
   \sum_{i=1}^{n} f(C_i) \leq f\left(\sum_{i=1}^{n}C_i\right)-\delta_f\sum_{i=1}^{n}\widetilde{C}_i\leq f\left(\sum_{i=1}^{n}C_i\right)
   \end{equation*}
 for all operators $C_i$ such that $0 \leq C_i\leq M \leq \sum_{i=1}^{n} C_i $ \ $(i=1,\ldots,n)$ for some scalar
 $M\geq0$, where $ \widetilde{C_i} = \frac{1}{2} - \left|\frac{C_i}{M}- \frac{1}{2} \right|$ and $\delta_f = f(0)+f(M) - 2 f\left(
 \frac{M}{2}\right)$.

\end{abstract} \maketitle
%-------------------------------------------------------------------------------
\section{Introduction and Preliminaries}
\noindent Let $\mathbb{B}(\mathscr{H})$ be the $C^*$-algebra of all
bounded linear operators on a complex Hilbert space $\mathscr{H}$
and $I$ denote the identity operator. If $\dim\mathscr{H}=n$, then
we identify $\mathbb{B}(\mathscr{H})$ with the $C^*$-algebra
$\mathcal{M}_n(\mathbb{C})$ of all $n\times n$ matrices with complex
entries. Let us endow the real space $\mathbb{B}_{h}(\mathscr{H})$
of all self-adjoint operators in $\mathbb{B}(\mathscr{H})$ with the
usual operator order $\leq$ defined by the cone of positive
operators of $\mathbb{B}(\mathscr{H})$.

If $T\in \mathbb{B}_{h}(\mathscr{H})$, then $m=\inf\{\langle
Tx,x\rangle: \|x\|=1\}$ and $M=\sup\{\langle Tx,x\rangle: \|x\|=1\}$
are called the bounds of $T$. We denote by $\sigma(J)$ the set of
all self-adjoint operators on $\mathscr{H}$ with spectra contained
in $J$. All real-valued functions are assumed to be continuous in
this paper. A real valued function $f$ defined on an interval $J$ is
said to be operator convex if $f(\lambda A+(1-\lambda)B)\leq \lambda
f(A)+(1-\lambda)f(B)$ for all $A,B\in\sigma(J)$  and all
$\lambda\in[0,1]$. If the function $f$ is operator convex, then the
so-called Jensen operator inequality $f(\Phi(A))\leq\Phi(f(A))$
holds for any unital positive linear map $\Phi$ on
$\mathbb{B}(\mathscr{H})$ and any $A\in\sigma(J)$. The reader is
referred to \cite{Fu, km, msm} for more information about operator
convex functions and other versions of the Jensen operator
inequality. It should be remarked that if $f$ is a real convex
function, but not operator convex, then the Jensen operator
inequality may not hold. To see this, consider the convex (but not
operator convex) function $f(t)=t^4$ defined on $[0,\infty)$ and the
positive mapping
$\Phi:\mathcal{M}_3(\mathbb{C})\to\mathcal{M}_2(\mathbb{C})$ defined
by $\Phi((a_{ij})_{1\leq i,j\leq 3}) = (a_{ij})_{1\leq i,j\leq 2}$
for any $A=(a_{ij})_{1\leq i,j\leq 3}\in\mathcal{M}_3(\mathbb{C})$.
If
 $$A=\left(
\begin{array}{ccc}
 2 &  2 &  0\\
   2  & 3 &  1\\
   0 &  1  & 3
\end{array}\right),$$
then there is no relationship between
\begin{align*}
 f(\Phi(A))=\left(
\begin{array}{cc}
36&46\\
46&59
\end{array}\right)\qquad\mbox{and}\qquad\Phi(f(A))=\left(
\begin{array}{cc}
36&48\\
48&68
\end{array}\right)
\end{align*}
in the usual operator order.

Recently, in \cite{MPP1} a version of the Jensen operator inequality was given without operator convexity as follows: \\
\textbf{Theorem A.} \cite[Theorem 1]{MPP1} Let $(A_1,\ldots,A_n)$ be
an $n$-tuple of operators $A_i \in \mathbb{B}_{h}(\mathscr{H})$ with
bounds $m_i$ and $M_i$, $m_i \leq M_i$, and let
$(\Phi_1,\ldots,\Phi_n)$ be an $n$-tuple of po\-si\-tive linear
mappings $\Phi_i$ on $\mathbb{B}(\mathscr{H})$ such that
$\sum_{i=1}^{n}\Phi_i (I)=I$. If
\begin{equation}\label{tA-condition}
(m_C,M_C) \cap [m_i,M_i]= {\O}
\end{equation}
for all $1 \leq i\leq n$, where $m_C$ and $M_C$ with $m_C \leq M_C$
are bounds of the self-adjoint operator
$C=\sum_{i=1}^{n}\Phi_{i}(A_i)$, then
\begin{equation}\label{tA-eq}
f\left(\sum_{i=1}^{n}\Phi_{i}(A_i)\right)\leq
\sum_{i=1}^{n}\Phi_i\left( f(A_i)\right)
\end{equation}
holds for every convex function $f:J\to\mathbb{R}$ provided that the
interval $J$ contains all $m_i,M_i$; see also \cite{MPPeric1}.

Another variant of the Jesnen operator inequality is the so-called
Jensen--Mercer operator inequality \cite{mpp} asserting that if $f$
is a real convex function on an interval $[m,M]$, then
\begin{eqnarray*}
 f\left(M+m-\sum_{i=1}^{n}\Phi_i(A_i)\right)\leq f(M)+f(m)-\sum_{i=1}^{n}\Phi_i(f(A_i)),
\end{eqnarray*}
where $\Phi_1,\cdots,\Phi_n$ are positive linear maps on
$\mathbb{B}(\mathscr{H})$ with $\sum_{i=1}^{n}\Phi_i(I)=I$ and
$A_1,\cdots,A_n\in\sigma([m,M])$.

Recently, in \cite{mmk} an extension of the Jensen--Mercer operator inequality was presented as follows:\\
\textbf{Theorem B.}\cite[Corollary 2.3]{mmk} Let $f$ be a convex
function on an interval $J$. Let $A_i,B_i,C_i,D_i\in\sigma(J)$\
$(i=1,\cdots,n)$ such that $A_i+D_i=B_i+C_i$ and $ A_i\leq m\leq
B_i,C_i\leq M\leq D_i $. Let $\Phi_1,\cdots,\Phi_n$ be positive
linear maps on $\mathbb{B}(\mathscr{H})$ with
$\sum_{i=1}^{n}\Phi_i(I)=I$.
 Then
\begin{align}\label{pre}
 f\left(\sum_{i=1}^{n}\Phi_i(B_i)\right)+
f\left(\sum_{i=1}^{n}\Phi_i(C_i)\right)\leq \sum_{i=1}^{n}\Phi_i(f(A_i))+\sum_{i=1}^{n}\Phi_i(f(D_i)).
\end{align}
\bigskip

%=================================================================
The authors of \cite{mmk} used inequality \eqref{pre} to obtain some operator inequalities. In particular, they gave a generalization of the Petrovi\'{c} operator inequality as follows:\\
\textbf{Theorem C.}\cite[Corollary 2.5]{mmk} Let
$A,D,B_i\in\sigma(J)$\, $ (i=1,\cdots,n)$ such that
$A+D=\sum_{i=1}^{n}B_i$ and $ A\leq m\leq B_i\leq M \leq D $\, $
(i=1,\cdots,n)$ for two real numbers $m<M$. If $f$ is convex on $J$,
then
 \begin{eqnarray*}
 \sum_{i=1}^{n}f(B_i)\leq (n-1)f\left(\frac{1}{n-1}A\right)+f(D).
 \end{eqnarray*}
%====================================================================
\bigskip

If $f:[0,\infty)\to \mathbb{R}$ is a convex function such that $f(0)= 0$, then
\begin{eqnarray}\label{sub1}
 f(a)+f(b)\leq f(a+b)
\end{eqnarray}
for all scalars $a,b\geq 0$. However, if the scalars $a,b$ are
replaced by two positive operators, this inequality may not hold.
For example if $f(t)=t^2$ and $A, B$ are the following two positive
matrices
$$
A=\left(
\begin{array}{ccc}
 1 &  0 &  0\\
   0  & 1 &  1\\
   0 &  1 &  2
\end{array}\right)\quad\mbox{and}\quad B=\left(
\begin{array}{ccc}
 1  & 1 &  1\\
   1 &  2  & 0\\
   1 &  0 &  2
\end{array}\right),
$$
then a straightforward computation reveals that there is no
relationship between $A^2+B^2$ and $(A+B)^2$ under the operator
order. Many authors tried to obtain some operator extensions of
\eqref{sub1}. In \cite{mn}, it was shown that
$$f(A+B)\leq f(A)+f(B)$$
for all non-negative operator monotone functions $f:[0,\infty)\to[0,\infty)$ if and only if $AB+BA$ is positive.

Another operator extension of \eqref{sub1} was established in \cite{mmk}\\
\textbf{Theorem D.} \cite[Corollary 2.9]{mmk} If
$f:[0,\infty)\to[0,\infty)$ is a convex function with $f(0)\leq0$,
then $  f(A)+f(B)\leq f(A+B) $
 for all invertible positive operators $A,B$ such that $A\leq MI\leq A+B$ and $B\leq MI\leq A+B$ for some scalar $M\geq0$.

Some other operator extensions of \eqref{sub1} can be found in
\cite{AZ,AA,U}. In this paper, as a continuation of \cite{mmk}, we
extend inequality \eqref{pre}, refine \eqref{pre} and improve some
of our results in \cite{mmk}. Some applications such as further
refinements of the Petrovi\'{c} operator inequality and the
Jensen--Mercer operator inequality are presented as well.

%-------------------------------------------------------------------------------
\section{Results}
To presenting our results, we introduce the abbreviation:
$$ \delta_f = f(m)+f(M) - 2 f\left( \frac{m+M}{2}\right)$$
for $f:[m,M]\rightarrow \mathbb{R}$, $m<M$.

\bigskip

We need the following lemma may be found in \cite[Lemma
2]{MPPeric1}. We give a proof for the sake of completeness.

\begin{lemma}\label{l-mpp}
Let $A \in \sigma([m,M])$, for some scalars $m < M$. Then
\begin{eqnarray}\label{l1-eq}
f\left(A \right) \leq \frac{M - A}{M-m} f(m) + \frac{A - m }{M-m}
f(M) - \delta_f \widetilde{A}
\end{eqnarray}
holds for every convex function $f:[m,M] \rightarrow \R$, where
$$
 \widetilde{A} = \frac{1}{2} - \frac{1}{M-m} \left|A- \frac{m+M}{2} \right|.
$$
 If $f$ is concave on $[m,M]$, then inequality \eqref{l1-eq} is reversed.
\end{lemma}
\begin{proof}
First assume that $a,b\in[m,M]$ and $\lambda\in[0,1/2]$ so that
$\lambda\leq 1-\lambda$. Then
\begin{align*}
 f(\lambda a+(1-\lambda)b)&=f\left(2\lambda\frac{a+b}{2}+(1-2\lambda)b\right)\\
&\leq 2\lambda f\left(\frac{a+b}{2}\right)+(1-2\lambda)f(b)\\
&=\lambda f(a)+(1-\lambda) f(b)-\lambda\left(f(a)+f(b)-2f\left(\frac{a+b}{2}\right)\right).
\end{align*}
It follows that
\begin{align}\label{e1} f(\lambda
a&+(1-\lambda)b)\nonumber\\
&\leq \lambda f(a)+(1-\lambda)f(b)-\min\{\lambda,1-\lambda\}
\left(f(a)+f(b)-2f\left(\frac{a+b}{2}\right)\right)
\end{align}
for all $a,b\in[m,M]$ and all $\lambda \in[0,1]$. If $t\in[m,M]$,
then by using \eqref{e1} with
  $\lambda=\frac{M-t}{M-m}$, $a=m$ and $b=M$ we obtain
\begin{align}\label{e2}
 f(t)&=f\left(\frac{M-t}{M-m}m+\frac{t-m}{M-m}M\right)\leq \frac{M-t}{M-m}f(m)+\frac{t-m}{M-m}f(M)\nonumber\\
&\quad -
\min\left\{\frac{M-t}{M-m},\frac{t-m}{M-m}\right\}
\left(f(m)+f(M)-2f\left(\frac{m+M}{2}\right)\right)
\end{align}
 for any $t\in[m,M]$. Since $\min\left\{\frac{M-t}{M-m},\frac{t-m}{M-m}\right\}=
\frac{1}{2}-\frac{1}{M-m}\left|t-\frac{m+M}{2}\right|$,
we have from \eqref{e2} that
\begin{align}\label{e3}
 f(t)&\leq \frac{M-t}{M-m}f(m)+\frac{t-m}{M-m}f(M)\nonumber\\
&\quad -\left(\frac{1}{2}-\frac{1}{M-m}\left|t-\frac{m+M}{2}\right|\right)
\left(f(m)+f(M)-2f\left(\frac{m+M}{2}\right)\right),
\end{align}
for all $t\in[m,M]$. Now if $A\in\sigma([m,M])$, then by utilizing
the functional calculus to \eqref{e3} we obtain \eqref{l1-eq}.
\end{proof}
%======================================================
In the next theorem we present a generalization of
\cite[Theorem~2.1]{mmk}.

\begin{theorem}\label{t0}
 Let $\Phi_i, \overline{\Phi}_i, \Psi_i, \overline{\Psi}_i$ be po\-si\-tive linear mappings on $\mathbb{B}(\mathscr{H})$ such that $\sum_{i=1}^{n_1}\Phi_i (I)=\alpha\, I$, $\sum_{i=1}^{n_2}\overline{\Phi}_i (I)=\beta\, I$, $\sum_{i=1}^{n_3}\Psi_i (I)=\gamma\, I$, $\sum_{i=1}^{n_4}\overline{\Psi}_i (I)=\delta\, I$ for some real numbers $\alpha, \beta, \gamma, \delta>0$. Let $A_i$\ $(i=1,\ldots,n_1)$, $D_i$\ $(i=1,\ldots,n_2)$,\ $C_i$\ $(i=1,\ldots,n_3)$ and $B_i$\ $(i=1,\ldots,n_4)$ be operators in $\sigma(J)$ such that $A_i\leq m\leq B_i,C_i\leq M\leq D_i$ for two real numbers $m < M$. If
\begin{equation}\label{ret0-condition}
  \frac{1}{\alpha}\sum_{i=1}^{n_1}\Phi_i(A_i) + \frac{1}{\delta}\sum_{i=1}^{n_2}\overline{\Phi}_i(D_i) = \frac{1}{\gamma}\sum_{i=1}^{n_3}\Psi_i(C_i) + \frac{1}{\beta}\sum_{i=1}^{n_4}\overline{\Psi}_i(B_i),
\end{equation}
then
\begin{align}\label{main}
 f\left(\frac{1}{\gamma}\sum_{i=1}^{n_3}\Psi_i(C_i)\right)+
  f\left(\frac{1}{\beta}\sum_{i=1}^{n_4}\overline{\Psi}_i(B_i)\right)
  &\leq \frac{1}{\alpha}\sum_{i=1}^{n_1}\Phi_i\left( f(A_i)\right) + \frac{1}{\delta} \sum_{i=1}^{n_2}\overline{\Phi}_i\left( f(D_i)\right)-\delta_f\widetilde{X}\nonumber\\
  &\leq
  \frac{1}{\alpha}\sum_{i=1}^{n_1}\Phi_i\left( f(A_i)\right) + \frac{1}{\delta} \sum_{i=1}^{n_2}\overline{\Phi}_i\left( f(D_i)\right)
\end{align}
holds for every convex function $f:J\to\mathbb{R}$, where
$$\widetilde{X}=1-\frac{1}{M-m}\left(\left|\frac{1}{\beta}\sum_{i=1}^{n_4}\overline{\Psi_i}(B_i)- \frac{m+M}{2} \right| +\left|\frac{1}{\gamma}\sum_{i=1}^{n_3}\Psi_i(C_i)- \frac{m+M}{2} \right|\right).$$
If $f$ is concave, then the reverse inequalities are
valid in \eqref{main}.
\end{theorem}
%=====
\begin{proof}
We prove only the case when $f$ is convex. Let $[m,M]\subseteq J$.
It follows from the convexity of $f$ on $J$ that
\begin{eqnarray}\label{re6}
 f(t)\geq \frac{M-t}{M-m}f(m)+\frac{t-m}{M-m}f(M)
\end{eqnarray}
for all $t\in J\setminus[m,M]$. Hence, by $A_i\leq m$ and $D_i\geq M$ we have
\begin{eqnarray}\label{re6}
 f(A_i)\geq \frac{M-A_i}{M-m}f(m)+\frac{A_i-m}{M-m}f(M)\quad (i=1,\cdots,n_1)
\end{eqnarray}
and similarly
\begin{eqnarray}\label{re7}
f(D_i)\geq \frac{M-D_i}{M-m}f(m)+\frac{D_i-m}{M-m}f(M)\quad (i=1,\cdots,n_2).
 \end{eqnarray}
Applying the positive linear mappings $\Phi_i$ and
$\overline{\Phi}_i$, respectively, to both sides of \eqref{re6} and
\eqref{re7} and summing we get
 \begin{eqnarray}\label{re8}
 \frac{1}{\alpha}\sum_{i=1}^{n_1}\Phi_i(f(A_i))\geq \frac{M- \frac{1}{\alpha}\sum_{i=1}^{n_1}\Phi_i(A_i)}{M-m}f(m)+
 \frac{\frac{1}{\alpha}\sum_{i=1}^{n_1}\Phi_i(A_i)-m}{M-m}f(M)
 \end{eqnarray}
and
\begin{eqnarray}\label{re9}
 \frac{1}{\delta}\sum_{i=1}^{n_2}\overline{\Phi}_i(f(D_i))\geq \frac{M- \frac{1}{\delta}\sum_{i=1}^{n_2}\overline{\Phi}_i(D_i)}{M-m}f(m)
 +\frac{\frac{1}{\delta}\sum_{i=1}^{n_2}\overline{\Phi}_i(D_i)-m}{M-m}f(M).
 \end{eqnarray}
On the other hand, taking into account that $m\leq \frac{1}{\beta}\sum_{i=1}^{n_4}\overline{\Psi_i}(B_i), \frac{1}{\gamma}\sum_{i=1}^{n_3}\Psi_i(C_i)\leq~ M$ and using Lemma~\ref{l-mpp} we obtain
 \small\begin{equation}\label{re3}
 f\left(\frac{1}{\beta}\sum_{i=1}^{n_4}\overline{\Psi_i}(B_i)\right)\leq \frac{M-\frac{1}{\beta}\sum_{i=1}^{n_4}\overline{\Psi_i}(B_i)}{M-m}f(m)+
 \frac{\frac{1}{\beta}\sum_{i=1}^{n_4}\overline{\Psi_i}(B_i)-m}{M-m}f(M) - \delta_f \widetilde{B}
 \end{equation}
and
\small\begin{equation}\label{re4}
 f\left(\frac{1}{\gamma}\sum_{i=1}^{n_3}\Psi_i(C_i)\right)\leq \frac{M-\frac{1}{\gamma}\sum_{i=1}^{n_3}\Psi_i(C_i)}{M-m}f(m)+
 \frac{\frac{1}{\gamma}\sum_{i=1}^{n_3}\Psi_i(C_i)-m}{M-m}f(M) - \delta_f \widetilde{C},
 \end{equation}
 where $\widetilde{B} =\frac{1}{2} - \frac{1}{M-m} \left|\frac{1}{\beta}\sum_{i=1}^{n_4}\overline{\Psi_i}(B_i)- \frac{m+M}{2} \right|$ and $\widetilde{C} =\frac{1}{2} - \frac{1}{M-m} \left|\frac{1}{\gamma}\sum_{i=1}^{n_3}\Psi_i(C_i)- \frac{m+M}{2} \right|$.
 \\
 Adding two inequalities \eqref{re3} and \eqref{re4} and putting $$\widetilde{X}=1-\frac{1}{M-m}\left(\left|\frac{1}{\beta}\sum_{i=1}^{n_4}\overline{\Psi_i}(B_i)- \frac{m+M}{2} \right| +\left|\frac{1}{\gamma}\sum_{i=1}^{n_3}\Psi_i(C_i)- \frac{m+M}{2} \right|\right)$$
 we obtain
 \begin{align*}
  & f\left(\frac{1}{\beta}\sum_{i=1}^{n_4}\overline{\Psi_i}(B_i)\right)+
  f\left(\frac{1}{\gamma}\sum_{i=1}^{n_3}\Psi_i(C_i)\right)\\
  &\leq \frac{2M-\frac{1}{\beta}\sum_{i=1}^{n_4}\overline{\Psi_i}(B_i)-
  \frac{1}{\gamma}\sum_{i=1}^{n_3}\Psi_i(C_i)}{M-m}f(m)\\
  &\quad +
 \frac{\frac{1}{\beta}\sum_{i=1}^{n_4}\overline{\Psi_i}(B_i)+
 \frac{1}{\gamma}\sum_{i=1}^{n_3}\Psi_i(C_i)-2m}{M-m}f(M) - \delta_f\widetilde{X}\\
 &= \frac{2M-\frac{1}{\alpha}\sum_{i=1}^{n_1}\Phi_i(A_i)-
  \frac{1}{\delta}\sum_{i=1}^{n_2}\overline{\Phi}_i(D_i)}{M-m}f(m)\\
  &\quad +
 \frac{\frac{1}{\alpha}\sum_{i=1}^{n_1}\Phi_i(A_i)+
 \frac{1}{\delta}\sum_{i=1}^{n_2}\overline{\Phi}_i(D_i)-2m}{M-m}f(M)- \delta_f\widetilde{X}\quad\mbox{(by \eqref{ret0-condition})}\\
 & \leq \frac{1}{\alpha}\sum_{i=1}^{n_1}\Phi_i(f(A_i))+ \frac{1}{\delta}\sum_{i=1}^{n_2}\overline{\Phi}_i(f(D_i))- \delta_f\widetilde{X},\quad \mbox{(by \eqref{re8} and \eqref{re9})}
  \end{align*}
which is the first inequality in \eqref{main}.

Furthermore, $m \leq \frac{1}{\beta}\sum_{i=1}^{n_4}\overline{\Psi_i}(B_i), \frac{1}{\gamma}\sum_{i=1}^{n_3}\Psi_i(C_i)\leq~ M $. The numerical inequality $\left|t-\frac{m+M}{2}\right|\leq\frac{M-m}{2}$\ $(m\leq t\leq M)$
 yields that $$ \left|\frac{1}{\beta}\sum_{i=1}^{n_4}\overline{\Psi_i}(B_i)- \frac{m+M}{2} \right| +\left|\frac{1}{\gamma}\sum_{i=1}^{n_3}\Psi_i(C_i)- \frac{m+M}{2} \right| \leq M-m.$$ Therefore $\widetilde{X}\geq 0$. Moreover, $f$ is convex on $[m,M]$. Hence $\delta_f \geq 0$. So the second inequality in \eqref{main} holds.
\end{proof}
%========================================================
\begin{remark}
 We can conclude some other versions of inequality \eqref{main}. In fact, under the assumptions in Theorem \ref{t0} the following inequalities hold true:
\small\begin{align*}
  &(1)\
\frac{1}{\gamma}\sum_{i=1}^{n_3}\Psi_i(f(C_i))+
\frac{1}{\beta}\sum_{i=1}^{n_4}\overline{\Psi_i}(f(B_i))
 \leq f\left(\frac{1}{\alpha}\sum_{i=1}^{n_1}\Phi_i(A_i)\right)+f\left(\frac{1}{\delta}
\sum_{i=1}^{n_2}\overline{\Phi_i}(D_i)\right) - \delta_f \widetilde{X}_2
  \\   & \hspace{6.4cm} \leq f\left(\frac{1}{\alpha}\sum_{i=1}^{n_1}\Phi_i(A_i)\right)+f\left(\frac{1}{\delta}
\sum_{i=1}^{n_2}\overline{\Phi_i}(D_i)\right);\\
&(2)\ f\left(\frac{1}{\gamma}
\sum_{i=1}^{n_3}\Psi_i(C_i)\right) + \frac{1}{\beta}\sum_{i=1}^{n_4}\overline{\Psi_i}(f(B_i))
\leq f\left(\frac{1}{\alpha}\sum_{i=1}^{n_1}\Phi_i(A_i)\right)+ \frac{1}{\delta}\sum_{i=1}^{n_2}\overline{\Phi}_i(f(D_i))- \delta_f \widetilde{X}_3
 \\   & \hspace{6.8cm} \leq
f\left(\frac{1}{\alpha}\sum_{i=1}^{n_1}\Phi_i(A_i)\right)+ \frac{1}{\delta}\sum_{i=1}^{n_2}\overline{\Phi}_i(f(D_i)),
 \end{align*}
in which
\begin{align*}
  \widetilde{X}_2 &=1-\frac{1}{M-m}\left[\frac{1}{\gamma}
\sum_{i=1}^{n_3}\Psi_i\left(\left|C_i-\frac{M+m}{2}\right|\right)+
\frac{1}{\beta}\sum_{i=1}^{n_4}
\overline{\Psi_i}\left(\left|B_i-\frac{M+m}{2}\right|\right)\right],\\
\widetilde{X}_3 &= 1-\frac{1}{M-m}\left[\left|\frac{1}{\gamma}
\sum_{i=1}^{n_3}\Psi_i(C_i)-\frac{M+m}{2}\right|+\frac{1}{\beta}\sum_{i=1}^{n_4}
\overline{\Psi_i}\left(\left|B_i-\frac{M+m}{2}\right|\right)
\right].
\end{align*}
\end{remark}
%=======================================================================
Before giving an example, we present some special cases of Theorem
\ref{t0} which are useful in our applications.
 The next corollary provides a refinement of \cite[Theorem~2.1]{mmk}.
%==================
\begin{corollary}\label{t1}
Let $f$ be a convex function on an interval $J$. Let
$A,B,C,D\in\sigma(J)$ such that $A+D=B+C$ and $ A\leq m\leq B,C\leq
M \leq D $ for two real numbers $m<M$.
 If $\Phi$ is a unital positive linear map on $\mathbb{B}(\mathscr{H})$, then
 \begin{align}\label{a}
 f(\Phi(B))+f(\Phi(C))&\leq \Phi(f(A))+\Phi(f(D)) - \delta_f \widetilde{X}\nonumber \\ & \leq \Phi(f(A))+\Phi(f(D)),
 \end{align}
 where
 $$
\widetilde{X} = 1 - \frac{1}{M-m} \left( \left|\Phi(B)- \frac{m+M}{2} \right| +\left|\Phi(C)- \frac{m+M}{2} \right| \right).$$
In particular,
\begin{align}\label{a1}
 f(B)+f(C)\leq f(A)+f(D)-\delta_f\widetilde{X}\leq f(A)+f(D).
\end{align}
 If $f$ is concave on $J$, then inequalities \eqref{a} and \eqref{a1} are reversed.
 \end{corollary}
%====================
Another special case of Theorem \ref{t0} leads to a refinement of \cite[Corollary~2.3]{mmk}.
%==================================
\begin{corollary}\label{co2}
 Let $f$ be a convex function on an interval $J$. Let $A_i,B_i,C_i,D_i\in\sigma(J)$ \ $(i=1,\cdots,n)$ such that $A_i+D_i=B_i+C_i$ and $ A_i\leq m\leq B_i, C_i\leq M\leq D_i $ \ $(i=1,\cdots,n)$. Let $\Phi_1,\cdots,\Phi_n$ be positive linear mappings on $\mathbb{B}(\mathscr{H})$ with $\sum_{i=1}^{n}\Phi_i(I)=I$.
 Then
 \begin{align*}
  &(1)\ f\left(\sum_{i=1}^{n}\Phi_i(B_i)\right)+f\left(\sum_{i=1}^{n}\Phi_i(C_i)\right) \leq  \sum_{i=1}^{n}\Phi_i(f(A_i))+\sum_{i=1}^{n}\Phi_i(f(D_i))- \delta_f \widetilde{X}_1
  \\ &\hspace{6.9cm}\leq \sum_{i=1}^{n}\Phi_i(f(A_i))+\sum_{i=1}^{n}\Phi_i(f(D_i));\\
 &(2)\ \sum_{i=1}^{n}\Phi_i(f(B_i))+\sum_{i=1}^{n}\Phi_i(f(C_i)) \leq f\left(\sum_{i=1}^{n}\Phi_i(A_i)\right)+f\left(\sum_{i=1}^{n}\Phi_i(D_i)\right) - \delta_f \widetilde{X}_2
  \\   & \hspace{6.2cm} \leq f\left(\sum_{i=1}^{n}\Phi_i(A_i)\right)+f\left(\sum_{i=1}^{n}\Phi_i(D_i)\right);\\
&(3)\ \sum_{i=1}^{n}\Phi_i(f(B_i))+f\left(\sum_{i=1}^{n}\Phi_i(C_i)\right) \leq f\left(\sum_{i=1}^{n}\Phi_i(D_i)\right)+ \sum_{i=1}^{n}\Phi_i(f(A_i))- \delta_f \widetilde{X}_3
 \\   & \hspace{6.2cm} \leq \displaystyle f\left(\sum_{i=1}^{n}\Phi_i(D_i)\right)+ \sum_{i=1}^{n}\Phi_i(f(A_i));
 \end{align*}
 where
 \begin{eqnarray*}
 \widetilde{X}_1 &=& \displaystyle 1 - \frac{1}{M-m} \left[\; \left|\sum_{i=1}^{n}\Phi_i(B_i)- \frac{m+M}{2} \right| +\left|\sum_{i=1}^{n}\Phi_i(C_i)- \frac{m+M}{2} \right| \;\right], \\
 \widetilde{X}_2 &=& \displaystyle 1 - \frac{1}{M-m} \left[\; \sum_{i=1}^{n}\Phi_i \left(\, \left|B_i- \frac{m+M}{2} \right| \,\right) +\sum_{i=1}^{n}\Phi_i \left(\, \left|C_i - \frac{m+M}{2} \right| \,\right) \;\right],\\
 \widetilde{X}_3 &=& \displaystyle 1 - \frac{1}{M-m} \left[\; \sum_{i=1}^{n}\Phi_i \left(\, \left|B_i- \frac{m+M}{2} \right| \,\right) +\left|\sum_{i=1}^{n}\Phi_i(C_i)- \frac{m+M}{2} \right| \;\right].
 \end{eqnarray*}
\end{corollary}
%===============================================
Now we give an example to show that how Theorem \ref{t0} works.
\begin{example}\label{ex-generally}
Let $n_i=1$ for $i=1,2,3,4$ and let $f(t)=t^4$. The function $f$ is convex but not operator convex\cite{Fu}. Let $\overline{\Phi}, \Psi, \overline{\Psi}=\Phi$ in which
$$\Phi : \mathcal{M}_3(\mathbb{C})
\rightarrow \mathcal{M}_2(\mathbb{C}), \ \ \Phi((a_{ij})_{1\leq
i,j\leq 3}) = (a_{ij})_{1\leq i,j\leq 2}. $$
 If

$$
A=\left(
\begin{array}{ccc}
1&-1&1 \\
-1&1&2 \\
1&2&-5
\end{array}\right), \,
D=\left(
\begin{array}{ccc}
9&1&1 \\
1& 10&2 \\
1&2&15
\end{array}\right), \,
C=\left(
\begin{array}{ccc}
4&1&2 \\
1& 4&1 \\
2&1&5
\end{array}\right), \,
B=\left(
\begin{array}{ccc}
6&-1&1 \\
-1&7&1 \\
1&1&5
\end{array}\right),
$$
then $ \Phi(A)+\Phi(D)=\Phi(C)+\Phi(B)$ and $A \leq 2.2 I\leq C, B
\leq 8I \leq D$. Also $\delta_f=2766.4$ and $\widetilde{X}=\left(
\begin{array}{cc}
0.655&0.345 \\
0.345&0.655 \\
\end{array}\right)$,
whence
\small\begin{align*}
&(\Phi(C))^{4}+(\Phi(B))^{4}=\left(
\begin{array}{cc}
1891&-859 \\
-859&3022 \\
\end{array}\right)\\
  &\lneqq
\left\{ \begin{array}{l}
 \left(\begin{array}{cc}
5281&2514.5 \\
2514.5&8758
\end{array}\right)=(\Phi(A))^{4}+(\Phi(D))^{4}-\delta_f\widetilde{X}\lneqq
    \left(\begin{array}{cc}
7093&3469 \\
3469&10570
\end{array}\right) = (\Phi(A))^{4}+(\Phi(D))^{4}\\[2ex]
\left(\begin{array}{cc}
5318&2576.5 \\
2576.5&8867
\end{array}\right)=\Phi(A^{4})+(\Phi(D))^{4}-\delta_f\widetilde{X}\lneqq
  \left(\begin{array}{cc}
7130&3531 \\
3531&10679
\end{array}\right) =  \Phi(A^{4})+(\Phi(D))^{4}\\[2ex]
\left(\begin{array}{cc}
6202&4311.5 \\
4311.5&12263
\end{array}\right) = (\Phi(A))^{4}+\Phi(D^{4})-\delta_f\widetilde{X}\lneqq
 \left(\begin{array}{cc}
8014&5266 \\
5266&14075
\end{array}\right) =  (\Phi(A))^{4}+\Phi(D^{4})\\[2ex]
 \left(\begin{array}{cc}
6239&4373.5 \\
4373.5&12372
\end{array}\right) =  \Phi(A^{4})+\Phi(D^{4})-\delta_f\widetilde{X}\lneqq
 \left(\begin{array}{cc}
8051&5328 \\
5328&14184
\end{array}\right) =  \Phi(A^{4})+\Phi(D^{4})
   \end{array} \right.
\end{align*}
This shows that inequalities in \eqref{main} can be strict.

Moreover,
\begin{eqnarray*}
 (\Phi(A))^{4}+\Phi(B^{4}) - \Phi(A^{4})-(\Phi(B))^{4} &=& \left(\begin{array}{cc}
884&1735 \\
1735&3396
\end{array}\right) \not \gtreqqless 0\\
(\Phi(A))^{4}+\Phi(B^{4}) - \Phi(A)^{4}-(\Phi(B))^{4} &=& \left(\begin{array}{cc}
921&1797 \\
1797&3505
\end{array}\right) \not \gtreqqless 0 \\
  \Phi(A^{4})+\Phi(B^{4}) - (\Phi(A))^{4}-\Phi(B^{4}) &=& \left(\begin{array}{cc}
921&1797 \\
1797&3505
\end{array}\right) \not \gtreqqless 0.
\end{eqnarray*}
Hence there is no relationship between the right hand sides of
inequalities in Corollary~\ref{co2}.
\end{example}
%===========================================================================

\begin{corollary}\label{t33}
Let $f$ be a convex function on an interval $J$.
 Let $A_i, B_i, C_i, D_i$, $i=1,\ldots,n$, be operators in $\sigma(J)$. If $A_i\leq m \leq C_i,B_i \leq M \leq D_i$, $i=1,\ldots,n$, for two real numbers $m < M$ and
\begin{equation}\label{t33-condition}
   \sum_{i=1}^{n} (A_i+ D_i) = \sum_{i=1}^{n} (C_i+B_i),
\end{equation}
then
\begin{eqnarray}\label{t33-eq1}
   \displaystyle f\left( \sum_{i=1}^{n} C_i \right)+ f\left(\sum_{i=1}^{n}  B_i\right)
   &\leq& \displaystyle f\left( \sum_{i=1}^{n} A_i \right)+ f\left( \sum_{i=1}^{n} D_i \right)-\delta_{f,n}\widetilde{X}_n, \nonumber\\
&\leq & f\left( \sum_{i=1}^{n} A_i \right)+ f\left( \sum_{i=1}^{n}
D_i \right)
 \end{eqnarray}
and
\begin{eqnarray}\label{t33-eq2}
   \displaystyle \sum_{i=1}^{n} f(C_i) + \sum_{i=1}^{n} f(B_i)
   &\leq& \displaystyle \sum_{i=1}^{n} f(A_i)+\sum_{i=1}^{n} f(D_i)-\delta_f\left(\sum_{i=1}^{n}(\widetilde{C_i}+\widetilde{B}_i)\right)\nonumber\\
&\leq &  \sum_{i=1}^{n} f(A_i)+\sum_{i=1}^{n} f(D_i)
    \end{eqnarray}
in which $\delta_{f,n}=f(nm)+f(nM)-2f\left(\frac{nm+nM}{2}\right)$ and
$$\widetilde{X}_n=1-\frac{1}{nM-nm}\left[\left|\sum_{i=1}^{n}C_i-\frac{nM+nm}{2}\right|
+\left|\sum_{i=1}^{n}B_i-\frac{nM+nm}{2}\right|\right].$$ If $f$ is
concave, then inequalities \eqref{t33-eq1} and \eqref{t33-eq2} are
reversed.
\end{corollary}
\begin{proof}
 We prove only inequality \eqref{t33-eq1} in the convex case.
 It follows from $ A_i\leq m \leq C_i,B_i \leq M \leq D_i$, $(i=1,\ldots,n)$ that
 $$ \sum_{i=1}^{n} A_i \leq mnI \leq \sum_{i=1}^{n} C_i, \, \sum_{i=1}^{n} B_i \leq MnI \leq \sum_{i=1}^{n} D_i.$$
 Using the same reasoning as in the proof of Theorem~\ref{t0} we get
 \begin{eqnarray*}
&& f\left( \sum_{i=1}^{n} C_i \right)+ f\left(\sum_{i=1}^{n}  B_i\right) \\
   &\leq& \frac{2Mn-\sum_{i=1}^{n}(C_i+B_i)}{Mn-mn}f(mn) + \frac{\sum_{i=1}^{n}(C_i+B_i)-2mn}{Mn-mn}f(Mn)-\delta_{f,n}\widetilde{X}_n \\
    &=& \frac{2Mn-\sum_{i=1}^{n}(A_i+D_i)}{Mn-mn}f(mn) + \frac{\sum_{i=1}^{n}(A_i+D_i)-2mn}{Mn-mn}f(Mn)-\delta_{f,n}\widetilde{X}_n \quad \text{(by \eqref{t33-condition})} \\
    &\leq& f\left( \sum_{i=1}^{n} A_i \right) + f\left( \sum_{i=1}^{n} D_i \right)-\delta_{f,n}\widetilde{X}_n,
 \end{eqnarray*}
 which give the first inequality in \eqref{t33-eq1}. It is easy to see that $\delta_{f,n}\widetilde{X}_n\geq0$, whence the second inequality derived.
\end{proof}
%=========================================================
\section{Applications}
Using the results in Section 2, we provide some applications which are refinements of some well-known operator inequalities.
As the first, we give a refinement of the operator Jensen--Mercer inequality.
\begin{corollary}
Let $\Phi_1,\cdots,\Phi_n$ be positive linear maps on
$\mathbb{B}(\mathscr{H})$ with $\sum_{i=1}^{n}\Phi_i(I)=I$ and
$B_1,\cdots,B_n\in\sigma([m,M])$ for two scalars $m<M$. If $f$ is a
convex function on $[m,M]$, then
\begin{eqnarray*}
 f\left(m+M-\sum_{i=1}^{n}\Phi_i(B_i)\right) & \leq& f(m)+f(M)-\sum_{i=1}^{n}\Phi_i(f(B_i)) - \delta_f \widetilde{B}
 \\   & \leq& \displaystyle f(m)+f(M)-\sum_{i=1}^{n}\Phi_i(f(B_i)),
\end{eqnarray*}
where $\widetilde{B} = \displaystyle 1 - \frac{1}{M-m} \left[\;
\sum_{i=1}^{n}\Phi_i \left(\, \left|B_i- \frac{m+M}{2} \right|
\,\right) +\left|\sum_{i=1}^{n}\Phi_i(B_i)- \frac{m+M}{2} \right|
\;\right].$
\end{corollary}
\begin{proof}
 Clearly $m\leq B_i\leq M$ \, $(i=1,\cdots,n)$. Set $C_i=M+m-B_i$ \ $(i=1,\cdots,n)$. Then $m\leq C_i\leq M$ and $B_i+C_i=m+M$\, $(i=1,\cdots,n)$. Applying inequality $(3)$ of Corollary \ref{co2} when $A_i=mI$ and $D_i=MI$ we obtain the desired inequalities.
 \end{proof}
 %======================================================================
The next result provides a refinement of the Petrovi\'c inequality for operators.
\begin{corollary}
If
$f:[0,\infty)\to\mathbb{R}$ is a convex function and
$B_1,\cdots,B_n$ are positive operators such that
$\sum_{i=1}^{n}B_i=MI$ for some scalar $M>0$, then
\begin{eqnarray*}
 \sum_{i=1}^{n}f(B_i)\leq f\left(\sum_{i=1}^{n}B_i\right)+(n-1)f(0) - \delta_f \widetilde{B}\leq f\left(\sum_{i=1}^{n}B_i\right)+(n-1)f(0),
 \end{eqnarray*}
 where
$\widetilde{B} = \displaystyle\frac{n}{2} - \sum_{i=1}^{n}
\left|\frac{B_i}{M}- \frac{1}{2} \right|.$
\end{corollary}
\begin{proof}
 It follows from $0\leq B_i\leq M$ that
$$f(B_i)\leq\frac{M-B_i}{M-0}f(0)+\frac{B_i-0}{M-0}f(M)-\delta_f\widetilde{B_i}
\quad\ (i=1,\cdots,n).$$ Summing above inequalities over $i$ we get
\begin{align*}
 \sum_{i=1}^{n}f(B_i)&\leq\frac{nM-\sum_{i=1}^{n}B_i}{M}f(0)+
\frac{\sum_{i=1}^{n}B_i}{M}f(M)-\delta_f\sum_{i=1}^{n}\widetilde{B_i}\\
&= (n-1)f(0)+f\left(\sum_{i=1}^{n}B_i\right)-\delta_f\widetilde{B}\ \quad (\mbox{by $\sum_{i=1}^{n}B_i=M$})\\
&\leq (n-1)f(0)+f\left(\sum_{i=1}^{n}B_i\right)\ \quad (\mbox{by $\delta_f\widetilde{B}\geq0$}).
\end{align*}
where $\widetilde{B} = \displaystyle\frac{n}{2} - \sum_{i=1}^{n}
\left|\frac{B_i}{M}- \frac{1}{2} \right|.$
\end{proof}
%==================================================================
As another consequence of Theorem~\ref{t0}, we present a refinement
of the Jensen operator inequality for real convex functions. The
authors of \cite{mmk} introduce a subset $\Omega$ of
$\mathbb{B}_{h}(\mathscr{H})\times\mathbb{B}_{h}(\mathscr{H})$
defined by
\begin{eqnarray*}
 \Omega=\left\{(A,B) \ \big| \ A\leq m\leq \frac{A+B}{2}\leq M\leq B, \ \ \mbox{for some} \ \ m,M\in\mathbb{R} \right\}.
\end{eqnarray*}
We have the following result.
%=========================================
\begin{corollary}\label{co7}
 Let $f$ be a convex function on an interval $J$ containing $m,M$. Let $\Phi_i$, $i=1,\ldots,n$, be  positive linear mappings on $\mathbb{B}(\mathscr{H})$ with $\sum_{i=1}^{n}\Phi_i(I)=I$. If $(A_i,D_i)\in\Omega$, $i=1,\ldots,n$, then
\begin{eqnarray}\label{co7-eq}
   \displaystyle f\left(\sum_{i=1}^{n}\Phi_i \left(\frac{A_i+D_i}{2} \right)\right)
    &\leq& \displaystyle \sum_{i=1}^{n}\Phi_i\left(\frac{f(A_i)+f(D_i)}{2}\right)-\delta_f\widetilde{X}\nonumber\\
&\leq& \sum_{i=1}^{n}\Phi_i\left(\frac{f(A_i)+f(D_i)}{2}\right),
\end{eqnarray}
where $$\widetilde{X}=\frac{1}{2}-\frac{1}{M-m}
\left|\sum_{i=1}^{n}\Phi_i\left(\frac{A_i+D_i}{2}\right)
-\frac{m+M}{2}\right|.$$ If $f$ is concave, then inequalities in
\eqref{co7-eq} are reversed.
\end{corollary}
\begin{proof}
 Putting $B_i=C_i=\frac{A_i+D_i}{2}$ and using inequality (1) of Corollary~\ref{co2}, we conclude the desired result.
\end{proof}
%========================
\bigskip
Note that utilizing Corollary~\ref{co2}, we even be able to obtain a
converse of the Jensen operator inequality. For this end, under the
assumptions in the Corollary~\ref{co7} we have
\begin{align}\label{ren}
\sum_{i=1}^{n}\Phi_i\left(f\left(\frac{A_i+D_i}{2}\right)\right)
&\leq\frac{1}{2}\left[f\left(\sum_{i=1}^{n}\Phi_i(A_i)\right)+
f\left(\sum_{i=1}^{n}\Phi_i(D_i)\right)\right]-\delta_f\widetilde{X}\nonumber\\
&\leq \frac{1}{2}\left[f\left(\sum_{i=1}^{n}\Phi_i(A_i)\right)+
f\left(\sum_{i=1}^{n}\Phi_i(D_i)\right)\right],
\end{align}
where $$\widetilde{X}=\frac{1}{2}-\frac{1}{M-m}\sum_{i=1}^{n}\Phi_i\left(
\left|\frac{A_i+D_i}{2}-\frac{m+M}{2}\right|\right).$$ Note that the
function $f$ need not to be operator convex. Let us give an example
to illustrate these inequalities.
%=============================================
\begin{example}
 Let $n=1$ and the unital positive linear map $\Phi : \mathcal{M}_3(\mathbb{C})\rightarrow~\mathcal{M}_2(\mathbb{C})$ be defined by
$$ \Phi((a_{ij})_{1\leq i,j\leq 3})= (a_{ij})_{1\leq i,j\leq 2}$$
 for each $A=(a_{ij})_{1\leq i,j\leq 3}\in\mathcal{M}_3(\mathbb{C})$. Consider the convex function $f(t)=e^t$ on $[0,\infty)$. If
\begin{align*}
 A=\left(
\begin{array}{ccc}
1&-1 & 0\\
-1&1&0\\
0&0&1
\end{array}\right)\qquad
D=\left(
\begin{array}{ccc}
7&-1 & 0\\
-1&6&0\\
0&0&7
\end{array}\right),
\end{align*}
then $0\leq A\leq 2I\leq \frac{A+D}{2}\leq 5I\leq D$, i.e., $(A,D)\in\Omega$.
Hence it follows from \eqref{co7-eq} that
\small\begin{align*}
 f\left(\Phi\left(\frac{A+D}{2}\right)\right)=\left(
\begin{array}{cc}
79.8&-50.5 \\
-50.5&54.6
\end{array}\right)&\lneqq
\left(\begin{array}{cc}
759.2&-399 \\
-399&344
\end{array}\right)=\Phi\left(\frac{f(A)+f(D)}{2}\right)-\delta_f\widetilde{X}\\
&\lneqq
\left(\begin{array}{cc}
768.2&-408 \\
-408&362
\end{array}\right)=\Phi\left(\frac{f(A)+f(D)}{2}\right),
\end{align*}
in which $\delta_f=89.6$ and $\widetilde{X}=\left(\begin{array}{cc}
0.1&0.1 \\
0.1&0.2
\end{array}\right)$.
\end{example}
\bigskip
It should be mentioned that in the case when $f$ is operator convex, under the assumptions in Corollary~\ref{co7} we have even more:
\begin{align*}
 f\left(\sum_{i=1}^{n}\Phi_i\left(\frac{A_i+D_i}{2}\right)\right)&\leq
\sum_{i=1}^{n}\Phi_i\left(f\left(\frac{A_i+D_i}{2}\right)\right)\quad(\mbox{by the Jensen inequality})\\
&\leq
\frac{1}{2}\left[f\left(\sum_{i=1}^{n}\Phi_i(A_i)\right)+
f\left(\sum_{i=1}^{n}\Phi_i(D_i)\right)\right]-\delta_f\widetilde{X}
\quad(\mbox{by \eqref{ren}})\\
&\leq
\frac{1}{2}\left[\sum_{i=1}^{n}\Phi_i(f(A_i)+f(D_i))\right]-\delta_f\widetilde{X}
\quad(\mbox{by the Jensen inequality})\\
&\leq
\sum_{i=1}^{n}\Phi_i\left(\frac{f(A_i)+f(D_i)}{2}\right)\quad(\mbox{since $\delta_f\widetilde{X}\geq0$}).
\end{align*}
%=====================================================================
\begin{corollary}\label{co4}
If $f$ is a convex function on an interval $J$ containing $m,M$, then
\begin{equation}\label{10}
\begin{array}{rcl}
 f(\lambda A+(1-\lambda)D) &\leq& \lambda f(A)+(1-\lambda)f(D) - \delta_f \widetilde{X} \\ &\leq& \lambda f(A)+(1-\lambda)f(D)
 \end{array}
\end{equation}
 for all $(A,D)\in\Omega$ and all $\lambda\in[0,1]$, where $\widetilde{X} = \displaystyle \frac{1}{2} - \frac{1}{M-m} \left|\frac{A+D- M-m}{2} \right|.$ If $f$ is concave, then inequality \eqref{10} is reversed.
 \end{corollary}
 \begin{proof}
Put $n=1$ and let $\Phi$ be the identity map in Corollary~\ref{co7} to get
 $$f\left(\frac{A+D}{2}\right)\leq\frac{f(A)+f(D)}{2}- \delta_f \widetilde{X} \leq\frac{f(A)+f(D)}{2}$$ for any $(A,D)\in\Omega$, which implies \eqref{10} by the continuity of $f$.
 \end{proof}
%=========================================================================
\bigskip

Regarding to obtain an operator version of \eqref{sub}, it is shown in \cite{mmk} that
 if $f:[0,\infty)\to[0,\infty)$ is a convex function with $f(0)\leq0$, then
 \begin{align}\label{sub}
 f(A)+f(B)\leq f(A+B)
 \end{align}
for all strictly positive operators $A,B$ for which $A\leq M\leq A+B$ and $B\leq M\leq A+B$ for some scalar $M$.
We give a refined extension of this result as follows.
%=======================================
\begin{theorem}\label{co77}
If $f:[0,\infty) \to \mathbb{R}$ is a convex function with $f(0)\leq
0$ then
   \begin{equation}\label{co77-eq}
   \sum_{i=1}^{n} f(C_i) \leq f\left(\sum_{i=1}^{n}C_i\right)-\delta_f\sum_{i=1}^{n}\widetilde{C}_i\leq f\left(\sum_{i=1}^{n}C_i\right)
   \end{equation}
for all positive operators $C_i$ such that $C_i\leq M \leq
\sum_{i=1}^{n} C_i $ \ $(i=1,\ldots,n)$ for some scalar $M\geq0$. If
$f$ is concave, then the reverse inequality is valid in
\eqref{co77-eq}.

 In particular, if $f$ is convex, then
$$f(A)+f(B)\leq f(A+B)-\delta_f\widetilde{X}\leq f(A+B) $$
 for all positive operators $A,B$ such that $A\leq MI\leq A+B$ and $B\leq MI\leq A+B$ for some scalar $M\geq0$, where
$\widetilde{X} = \displaystyle 1 -  \left|\frac{A}{M}- \frac{1}{2}
\right| -  \left|\frac{B}{M}- \frac{1}{2} \right|.$
\end{theorem}
\begin{proof}
Without loss of generality let $M>0$. Lemma \ref{l-mpp} implies that
$$f(C_i) \leq \frac{MI-C_i}{M-0}f(0)+\frac{C_i}{M-0}f(M)-\delta_f\widetilde{C}_i
=\frac{C_i}{M}f(M)-\delta_f\widetilde{C_i} \quad (i=1,\ldots,n) $$ since $f(0)\leq0$.
Summing the above inequalities over $i$ we get
\begin{equation}\label{co77-eq1}
\sum_{i=1}^{n} f(C_i) \leq \frac{\sum_{i=1}^{n}C_i}{M}f(M)-
\delta_f\sum_{i=1}^{n}\widetilde{C}_i.
\end{equation}
Since the spectrum of $\sum_{i=1}^{n} C_i$ is contained in $[M,\infty)\subset[0,\infty) \setminus [0,M)$, we have
\begin{align}\label{co77-eq2}
f\left(\sum_{i=1}^{n} C_i \right) &\geq \frac{MI-\sum_{i=1}^{n}C_i}{M-0}f(0)+\frac{\sum_{i=1}^{n}C_i}{M-0}f(M)\nonumber\\
&\geq
\frac{\sum_{i=1}^{n}C_i}{M}f(M)\quad(\mbox{since $MI\leq\sum_{i=1}^{n}C_i$ and $f(0)\leq0$}).
\end{align}
 Combining two inequalities \eqref{co77-eq1} and \eqref{co77-eq2}, we
reach to the desired inequality \eqref{co77-eq}.
\end{proof}
%============================================
%==================================================================

\begin{theorem}\label{ad}
 Let $A,B,C,D\in\sigma(J)$ such that $ A\leq m\leq B,C\leq M \leq D $ for two real numbers
$m<M$. If f is a convex function on $J$ and any one of the following conditions
\begin{eqnarray*}
 {\rm(i)} \ \ B+C\leq A+D \quad\mbox{and}\quad f(m)\leq f(M)\\
 {\rm(ii)} \ \ A+D\leq B+C\quad\mbox{and}\quad f(M)\leq f(m)
\end{eqnarray*}
is satisfied, then
\begin{eqnarray}\label{ad1}
 f(B)+f(C)\leq f(A)+f(D)- \delta_f \widetilde{X} \leq f(A)+f(D),
\end{eqnarray}
where
$\widetilde{X} = \displaystyle 1 - \frac{1}{M-m} \left( \left|B- \frac{M+m}{2} \right| + \left|C- \frac{M+m}{2} \right| \right).$

 If f is concave and any one of the following conditions
\begin{eqnarray*}
 {\rm(iii)} \ \ B+C\leq A+D \quad\mbox{and}\quad f(M)\leq f(m)\\
 {\rm(iv)} \ \ A+D\leq B+C\quad\mbox{and}\quad f(m)\leq f(M)
\end{eqnarray*}
is satisfied, then inequality \eqref{ad1} is reversed.
\end{theorem}
\begin{proof}
Let $f$ be convex and ${\rm(i)}$ is valid. It follows from Lemma~\ref{l-mpp}
that
 $$f(B)\leq \frac{f(M)-f(m)}{M-m}B+\frac{f(m)M-f(M)m}{M-m} - \delta_f \left( \frac{1}{2} - \frac{1}{M-m} \left|B- \frac{M+m}{2} \right| \right)$$
 and
 $$f(C)\leq \frac{f(M)-f(m)}{M-m}C+\frac{f(m)M-f(M)m}{M-m}- \delta_f \left( \frac{1}{2} - \frac{1}{M-m} \left|C- \frac{M+m}{2} \right| \right).$$
 Summing above inequalities we get
 \begin{align*}
 f(B)+f(C)&\leq \frac{f(M)-f(m)}{M-m}(B+C)+2\ \frac{f(m)M-f(M)m}{M-m} -\delta_f \widetilde{X}\\
 &\leq \frac{f(M)-f(m)}{M-m}(A+D)+2\ \frac{f(m)M-f(M)m}{M-m}-\delta_f \widetilde{X} \quad (\mbox{by {\rm(i)}})\\
 &=\frac{f(M)-f(m)}{M-m}A+\frac{f(m)M-f(M)m}{M-m}\\
 &\ \ + \frac{f(M)-f(m)}{M-m}D+\frac{f(m)M-f(M)m}{M-m} - \delta_f \widetilde{X}\\
 & \leq f(A)+f(D)- \delta_f \widetilde{X}\qquad (\mbox{by \eqref{re6} and \eqref{re7}}) \\
 & \leq f(A)+f(D)\qquad (\mbox{by $\delta_f \widetilde{X}\geq 0$ })
 \end{align*}
 The other cases can be verified similarly.
\end{proof}
%===========================================
Applying the above theorem to the power functions we get
\begin{corollary}
 Let $A,B,C,D\in \mathbb{B}_{h}(\mathscr{H})$ be such that $ I\leq A\leq m\leq B,C\leq M \leq D $ for two real numbers $m< M$. If one of the following conditions \begin{eqnarray*}
 {\rm(i)} \ \ B+C\leq A+D \quad\mbox{and}\quad p\geq 1\\
 {\rm(ii)} \ \ A+D\leq B+C\quad\mbox{and}\quad p\leq0
\end{eqnarray*}
is satisfied, then
 \begin{eqnarray*}
B^p+C^p\leq A^q+D^q - \delta_p \widetilde{X} \leq A^q+D^q
\end{eqnarray*}
 for each $q\geq p$, where $$\displaystyle \delta_p = m^p+M^p - 2 \left( \frac{m+M}{2}\right)^p, \, \,
\displaystyle \widetilde{X} = 1 - \frac{1}{M-m} \left( \left|B- \frac{M+m}{2} \right| + \left|C- \frac{M+m}{2} \right| \right).$$
\end{corollary}
\begin{proof}
Let ${\rm(i)}$ be valid. Applying Theorem~\ref{ad} for $f(t)=t^p$, it follows
 \begin{align*}
 B^p+C^p&\leq A^p+D^p- \delta_p \widetilde{X}\\
 &\leq A^q+D^q- \delta_p \widetilde{X} \qquad (\mbox{by $q\geq p$})\\
 &\leq A^q+D^q \qquad (\mbox{by $\delta_p \widetilde{X}\geq 0$ })
 \end{align*}
 The other cases may be verified similarly.
\end{proof}
%=======================================================================================

\textbf{Acknowledgement.} The third author would like to thank the
Tusi Mathematical Research Group (TMRG), Mashhad, Iran.

\bibliographystyle{amsplain}

\end{document}